\documentclass[english, final, a4paper, authoryear]{elsarticle}

\usepackage[T1]{fontenc}
\usepackage{lmodern}
\usepackage{martin}
\usepackage[mathcal]{euscript}
\usepackage{longtable}
\usepackage{multirow}
\usepackage{amssymb}
\usepackage{pifont}

\begin{document}

\title{Critical configurations for two projective views, a new approach}
\author{Martin Bråtelund}
\address{Department of Mathematics, University of Oslo, Moltke Moes vei 35, 0316 Oslo Norway\\
mabraate@mail.uio.no}
\date{\today}

\begin{abstract}
This article develops new techniques to classify critical configurations for 3D scene reconstruction from images taken by unknown cameras. Generally, all information can be uniquely recovered if enough images and image points are provided, but there are certain cases where unique recovery is impossible; these are called \emph{critical configurations}. In this paper, we use an algebraic approach to study the critical configurations for two projective cameras. We show that all critical configurations lie on quadric surfaces, and classify exactly which quadrics constitute a critical configuration. The paper also describes the relation between the different reconstructions when unique reconstruction is impossible.
\end{abstract}
\maketitle


\setlength{\parindent}{0 pt }
\setlength{\parskip}{2.25 ex plus 0.75 ex minus 0.3 ex }

\section{Introduction}
In computer vision, one of the main problems is that of \emph{structure from motion}, where given a set of $2$-dimensional images the task is to reconstruct a scene in $3$-space and find the camera positions in this scene. Over time, many techniques have been developed for solving these problems for varying camera models and under different assumptions on the space and image points \cite{maybank1993theory, hartleyzisserman}. In general, with enough images and enough points in each image, one can uniquely recover all information about the original scene. However, there are also some configurations of points and images where a unique recovery is never possible. These are called \emph{critical configurations}.

Much work has been done to understand critical configurations for various settings \cite{buchanan1988twistedcubic, hartley2000, hartleyKahlAstrom2001, HKCalibrated, bertolini2007criticalOneView, HK, bertolini2020criticalRank, BunchananCriticalLines}, with results dating as far back as 1941 \cite{Krames1941}. While interesting from a purely theoretical viewpoint, critical configurations also play a part in practical applications. Even though critical configurations are rare in real-life reconstruction problems when enough data is available (due to noise), it has been shown that as the configurations approach the critical ones, reconstruction algorithms tend to become less and less stable \cite{stability, HK, bertoliniStability}.

Our new techniques confirm the main result of \cite{HK} for two cameras, and provide proofs for many unproven assertions in \cite{HK} (see Section 4). Moreover, they yield explicit maps between different critical configurations that have the same images (see Section 5). We also give a complete description of all critical configurations in the case of a single camera (see Section 3). Our companion paper \cite{threeViews} uses the techniques developed in this article to correct the classification of \cite{HK} for three cameras. We plan to use these new techniques to classify critical configurations for any number of views, as well as using them in other, more complicated scenarios (e.g., cameras observing lines \cite{felixLineMVI, BunchananCriticalLines}, or rolling-shutter cameras \cite{rollingShutterDegeneracies}) in future work. 

The main result of this paper is the classification of the critical configurations for two views in \cref{thr:critical_conf_for_two_views}.

\section{Background}
\label{sec:background}
We refer the reader to \cite{hartleyzisserman} for the basics on computer vision and multi-view geometry.

Let $\mathbb{C}$ denote the complex numbers, and let $\p{n}$ denote the projective space over the vector space $\mathbb{C}^{n+1}$. Projection from a point $p\in\p3$ is a linear map
\begin{equation*}
P\from \p3\dashrightarrow\p2.
\end{equation*}
We refer to such a projection and its projection center $p$ as a \emph{camera} and its \emph{camera center} (following established terminology, we use the words \emph{camera} and \emph{view} interchangeably). Following this theme, we refer to points in $\p3$ as \emph{space points} and points in $\p2$ as \emph{image points}. Similarly, $\p2$ will be referred to as an \emph{image}.

Once a basis is chosen in $\p3$ and $\p2$, a camera can be represented by a $3\times4$ matrix of full rank called the \emph{camera matrix}. The camera center is then given as the kernel of the matrix. For the most part, we make no distinction between a camera and its camera matrix, referring to both simply as cameras. We use the \emph{real projective pinhole camera model}, meaning that we require a camera matrix to be of full rank and to have only real entries.

\begin{remark}
\label{rem:basis_chosen}
Throughout the paper, whenever we talk about cameras it is to be understood that a choice of basis has been made, both on the images and 3-space.
\end{remark}

Since the map $P$ is not defined at the camera center $p$, it is not a morphism. This problem can be mended by taking the blow-up. Let $\widetilde{\p3}$ be the blow-up of $\p3$ in the camera center of $P$. We then get the following diagram:
\begin{equation*}
\begin{tikzcd}
\widetilde{\p3} \arrow[r, "P"] \arrow[d, "\pi"'] & \p2 \\
\p3 \arrow[ru, "P"', dashed]                     &    
\end{tikzcd}
\end{equation*}
where $\pi$ denotes the blow-down of $\widetilde{\p3}$. This gives a morphism from $\widetilde{\p3}$ to $\p2$. For ease of notation, we retain the symbol $P$ and the names \emph{camera} and \emph{camera center}, although one should note that in $\widetilde{\p3}$ the camera center is no longer a point, but an exceptional divisor. 

\begin{definition}
Given an $n$-tuple of cameras $\textbf{P}=(P_1,\ldots,P_n)$, with camera centers $p_1,\ldots,p_n$, let $\widetilde{\p3}$ denote the blow-up of $\p3$ in the $n$ camera centers. We define the \emph{joint camera map} to be the map
\begin{align*}
\widetilde{\phi_\textbf{P}}=P_1\times\cdots\times P_n\from\widetilde{\p3}&\to(\p2)^{n},\\
x&\mapsto (P_1(x),\ldots ,P_n(x)).
\end{align*}
\end{definition}
\begin{remark}
Throughout the paper, we assume that all camera centers are distinct.
\end{remark}
This again gives a commutative diagram
\begin{equation*}
\begin{tikzcd}
\widetilde{\p3} \arrow[r, "\widetilde{\phi_\textbf{P}}"] \arrow[d, "\pi"'] & (\p2)^{n} \\
\p3 \arrow[ru, "\phi_\textbf{P}"', dashed]                     &    
\end{tikzcd}
\end{equation*}
The reason we use the blow-up $\widetilde{\p3}$ rather $\p3$ is to turn the cameras (and hence the joint camera map) into morphisms rather than rational maps. This ensures that the image of the joint camera map is Zariski closed, turning it into a projective variety. 

\begin{definition}
We denote the image of the joint camera map $\widetilde{\phi_\textbf{P}}$ as the \emph{multi-view variety} of $P_1,\ldots,P_n$. The set of all homogeneous polynomials vanishing on $\im(\widetilde{\phi_\textbf{P}})$ is an ideal that we denote as the \emph{multi-view ideal}.
\end{definition}

\textbf{Notation.} Most works on computer vision do not use the blow-up, defining cameras/the joint camera map as rational maps rather than morphisms. Hence, they tend to define the multi-view variety as the \emph{closure} of the image rather than as the image itself. While our definition of the multi-view variety seems different, it is equivalent to that used in other works, like \cite{Tomas}.

\textbf{Notation.} While the multi-view variety is always irreducible, we use the term \emph{variety} to also include reducible algebraic sets.

\begin{definition}
Given a set of points $S\subset(\p2)^{n}$, a \emph{reconstruction} of $S$ is an $n$-tuple of cameras $\textbf{P}=(P_1,\ldots,P_n)$ and a set of points $X\subset \widetilde{\p3}$ such that $S=\widetilde{\phi_\textbf{P}}(X)$ where $\widetilde{\phi_\textbf{P}}$ is the joint camera map of the cameras $P_1,\ldots,P_n$. 
\end{definition}

\begin{definition}
Given a configuration of cameras and points $(P_1,\ldots,P_n,X)$, we refer to $\widetilde{\phi_\textbf{P}}(X)\subset(\p2)^{n}$ as the \emph{images} of $(P_1,\ldots,P_n,X)$.
\end{definition}

Given a set of image points $S\subset(\p2)^{n}$ as well as a reconstruction $(P_1,\ldots,P_n,X)$ of $S$, note that any scaling, rotation, translation, or more generally, any real projective transformation of $(P_1,\ldots,P_n,X)$ does not change the images, giving rise to a large family of reconstructions of $S$. However, we are not interested in differentiating between these reconstructions.

\begin{definition}
\label{def:equvalent_configurations}
Given a set of points $S\subset(\p2)^{n}$, let $(\textbf{P},X)$ and $(\textbf{Q},Y)$ be two reconstructions of $S$, let $\overline{X}$ and $\overline{Y}$ denote the blow-downs of $X$ and $Y$ respectively and let $P_i'$ and $Q_i'$ be the matrix representation of $P_i$ and $Q_i$ respectively. The two reconstructions of $S$ are considered \emph{equivalent} if there exists an element $A\in\PGL(4)$, such that
\begin{align*}
&A(\overline{X})=\overline{Y},\\
&P_i'A^{-1}=Q_i',\quad \forall i.
\end{align*}
\end{definition}

From now on, whenever we talk about a configuration of cameras and points, it is to be understood as unique up to such an isomorphism/action of $\PGL(4)$, and two configurations are considered different only if they are not isomorphic/do not lie in the same orbit under this group action. As such, we consider a reconstruction to be unique if it is unique up to action by $\PGL(4)$.

\begin{definition}
Given a configuration of cameras and points $(P_1,\ldots,P_n,X)$, a \emph{conjugate configuration} is a configuration $(Q_1,\ldots,Q_n,Y)$, nonequivalent to the first, such that $\widetilde{\phi_\textbf{P}}(X)=\widetilde{\phi_\textbf{Q}}(Y)$. Pairs of points $(x,y)\in X\times Y$ are called \emph{conjugate points} if $\widetilde{\phi_\textbf{P}}(x)=\widetilde{\phi_\textbf{Q}}(y)$.
\end{definition}

\begin{definition}
\label{def:conjugate_configuration/point}
A configuration of cameras $(P_1,\ldots,P_n)$ and points $X\subset \widetilde{\p3}$ is said to be a \emph{critical configuration} if it has at least one conjugate configuration. A critical configuration $(P_1,\ldots,P_n,X)$ is said to be \emph{maximal} if there exists no critical configuration $(P_1,\ldots,P_n,X')$ such that $X\subsetneq X'$.
\end{definition}

Hence, a configuration is critical if and only if the images it produces do not have a unique reconstruction.

\begin{remark}
Various definitions of critical configurations exist. For instance, \cite{Krames1941} considers the cone with two cameras on the same generator to be critical, while it fails to be critical by our definition. We use a definition similar to the one in \cite{HK}, except we are working in $\widetilde{\p3}$. If one considers the blow-down, our definition matches the results in \cite{HK}.
\end{remark}

\begin{definition}
\label{def:set_of_critical_points}
Let $\textbf{P}$ and $\textbf{Q}$ be two $n$-tuples of cameras, let $\widetilde{\p3}_P$ and $\widetilde{\p3}_Q$ denote the blow-up of $\p3$ in the camera centers of $\textbf{P}$ and $\textbf{Q}$ respectively. Projecting the fiber product
\begin{align*}
\widetilde{\p3}_P\times_{\p2}\widetilde{\p3}_Q=\Set{(x,y)\in\widetilde{\p3}_P\times\widetilde{\p3}_Q\mid \widetilde{\phi_\textbf{P}}(x)=\widetilde{\phi_\textbf{Q}}(y)}
\end{align*}
to the first coordinate gives a variety, $X$. We call $X$ the \emph{set of critical points of $P$ with respect to $Q$}.
\end{definition}

This definition is motivated by the following fact:

\begin{proposition}
\label{prop:critical_points_give_critical_configuration}
Let $\textbf{P}$ and $\textbf{Q}$ be two (different) $n$-tuples of cameras, and let $X$ and $Y$ be their respective sets of critical points. Then
$(\textbf{P},X)$ is a critical configuration, with $(\textbf{Q},Y)$ as its conjugate. Furthermore, $(\textbf{P},X)$ is maximal with respect to $\textbf{Q}$ in the sense that if there exists a critical configuration $(\textbf{P},X')$ with $X\subsetneq X'$ then its conjugate consists of cameras different from $\textbf{Q}$
\end{proposition}
\begin{proof}
It follows from \cref{def:set_of_critical_points} that for each point $x\in X$, we have a conjugate point $y\in Y$. Hence the two configurations have the same images, so they are both critical configurations, conjugate to one another. 

The (partial) maximality follows from the fact that if we add a point $x_0$ to $X$ that does not lie in the set of critical points, there is (by \cref{def:set_of_critical_points}) no point $y_0\in\widetilde{\p3}_Q$ such that $\widetilde{\phi_\textbf{P}}(x_0)=\widetilde{\phi_\textbf{Q}}(y_0)$. Hence $(\textbf{P},X)$ will no longer be critical.
\end{proof}

For two different pairs of cameras, the sets of critical points turn out to be the quadric surfaces $\widetilde{S_P},\widetilde{S_Q}$ described in \cref{sec:preliminary_two_views}.

The goal of this paper is to classify all maximal critical configurations for three cameras. The reason we focus primarily on the maximal ones is that every critical configuration is contained in a maximal one and (when working with more than one camera) the converse is true as well, any subconfiguration of a critical configuration is itself critical.

We conclude this section with a final, useful property of critical configurations, namely that the only property of the cameras we need to consider when exploring critical configurations is the position of their camera centers (i.e. change of coordinates in the images does not affect criticality). 

\begin{proposition}[{\cite[Proposition 3.7]{HK}}]
\label{prop:only_camera_centers_matter}
Let $(P_1,\ldots,P_n)$ be $n$ cameras with centers $p_1,\ldots,p_n$, and let $(P_1,\ldots,P_n,X)$ be a critical configuration. \newline If $(P_1',\ldots,P_n')$ is a set of cameras sharing the same camera centers, the configuration $(P_1',\ldots,P_n',X)$ is critical as well.
\end{proposition}
\begin{proof}
Since $P_i$ and $P_i'$ share the same camera center and the camera center determines the map uniquely up to a choice of coordinates, there exists some $H_i\in\PGL(3)$ such that $P_i'=H_iP_i$. Let $(Q_1,\ldots,Q_n,Y)$ be a conjugate to $(P_1,\ldots,P_n,X)$. Then $(H_1Q_1,\ldots,H_nQ_n,Y)$ is a conjugate to $(H_1P_1,\ldots,H_nP_n,X)=(P_1',\ldots,P_n',X)$, so this configuration is critical as well.
\end{proof}

\section{The one-view case}
\label{sec:one_view_case}
Reconstruction of a 3D-scene from the image of one projective camera is generally considered impossible, so most papers start with the two-view case. Still, for the sake of completeness, we give a summary of the critical configurations for one camera. 

Let $P$ be a camera, and let $p$ be its camera center, we then have the joint camera map:
\begin{align*}
\widetilde{\phi_P}\from\widetilde{\p3}\to\p2
\end{align*}
For any point $x\in\p2$, the fiber over $x$ is a line through $p$, so no point can be uniquely recovered. From this, one might assume that every configuration with one camera is critical. However, this is only the case if our configuration consists of sufficiently many points.

\begin{theorem}
A configuration of one point and one camera is never critical. A configuration of one camera and $n>1$ points is critical if and only if the camera center along with the $n$ points span a space of dimension less than $n$.
\end{theorem}

\begin{proof}
For the first part, note that up to a projective transformation, there exists only one configuration of one point and one camera. In other words, any configuration of one point and one camera can be taken to any other such configuration by simply changing coordinates. By \cref{def:equvalent_configurations} this makes them equivalent, which means that only one reconstruction exists. 

The same turns out to be the case if the configuration is such that the camera center along with the $n>1$ points span a space of dimension $n$. In $\p3$, one can never span a space of dimension greater than 3, so this implies that $n\leq3$. Furthermore, if the $n$ points along with the camera center span a space of dimension $n$,  then the points and camera center lie in general position. However, for $n\leq4$ points (fewer than 3 points + one camera center) there exists only one configuration (up to action with $\PGL(4)$) where all points are in general position. This means (by \cref{def:equvalent_configurations}) that all reconstructions are equivalent.

Now it only remains to show that a configuration is critical if the camera center along with the $n>1$ points span a space of dimension less than $n$. Indeed, if the points along with the camera center span a space of dimension less than $n$, the image points span a space of dimension $m<n-1$. Then there are at least two nonequivalent reconstructions; one reconstruction where the points span a space of dimension $m$ not containing the camera center and one where they span a space of dimension $m+1$ which contains the camera center. 
\end{proof}

\section{The two-view case}
\label{sec:preliminary_two_views}

\subsection{The multi-view ideal}
We start the study of the case of two cameras $P_1$ and $P_2$ by understanding their multi-view variety. We assume, here and throughout the rest of the paper, that all cameras have distinct centers. The two cameras define the joint camera map:
\begin{align*}
\widetilde{\phi_P}\from \widetilde{\p3}\to\pxp.
\end{align*}

\begin{proposition}
\label{lem:joint-camera_map_is_isomorphism}
For two cameras $(P_1,P_2)$, the joint camera map $\widetilde{\phi_P}$ takes the line spanned by the two camera centers to a point, and is an embedding everywhere else.
\end{proposition}
\begin{proof}
For $x=(x_1,x_2)\in\pxp$, the preimage of $x$ is given by
\begin{align*}
\widetilde{\phi_P}^{-1}(x)=l_1\cap l_2,
\end{align*}
where $l_i$ is the line $P_i^{-1}(x_1)$. The line $l_i$ passes through the camera center $p_i$, so the intersection of the two lines is a single point unless they are both equal to the line spanned by the camera centers.
\end{proof}

This means that the multi-view variety $\im(\widetilde{\phi_P})$ is an irreducible singular $3$-fold in $\pxp$. It is described by a single bilinear polynomial $F_P$, which we call \emph{the fundamental form} of $P_1$, $P_2$.

The fundamental form is well-studied in the literature and is often represented by a $3\times3$ matrix of rank $2$ called the \emph{fundamental matrix}. See \cite[section~9.2]{hartleyzisserman} for a geometric construction of the fundamental matrix. We use the construction in \cite{Tomas} where the fundamental form is given as the determinant of a $6\times6$ matrix:

\subsection{The multi-view variety}
\begin{proposition}[The fundamental form]{\cite[Sections 9.2 and 17.1]{hartleyzisserman}}
\label{lem:fundamental_form}
For two cameras $P_1,P_2$, the multi-view variety $\im(\widetilde{\phi_\textbf{P}})\subset\pxp$ is the vanishing locus of a single, bilinear, rank 2 form $F_P$, called the fundamental form (or fundamental matrix).
\begin{align*}
F_P(\textbf{x},\textbf{y})=\det\begin{bmatrix}
P_i&\textbf{x}&\textbf{0}\\
P_j&\textbf{0}&\textbf{y}
\end{bmatrix},
\end{align*}
where $\textbf{x}$ and $\textbf{y}$ are the variables in the first and second image respectively.
\end{proposition}

\begin{proof}
By \cref{lem:joint-camera_map_is_isomorphism}, the multi-view variety for two cameras is an irreducible, 3-fold. It follows that the multi-view ideal is generated by a single polynomial. Let $(\textbf{x},\textbf{y})$ be a generic point in the multi-view variety, then there exists a point $\textbf{X}\in\p3$ such that $P_1(\textbf{X})=\lambda_1\textbf{x}$ and $P_2(\textbf{X})=\lambda_2\textbf{y}$, then
\begin{align*}
\begin{bmatrix}
P_i&\textbf{x}&\textbf{0}\\
P_j&\textbf{0}&\textbf{y}
\end{bmatrix}\begin{bmatrix}
\textbf{X}\\
-\lambda_1\\
-\lambda_2
\end{bmatrix}=0.
\end{align*}
Since the matrix has a non-zero kernel, the determinant $F_P$ has to vanish on $(\textbf{x},\textbf{y})$. Now we need only show that the determinant is irreducible to prove that $F_P(\textbf{x},\textbf{y})$ generates the multi-view ideal. Irreducibility follows from the fact that the polynomial is of rank 2 (a reducible polynomial is always of rank 1) which in turn follows from the fact that $F_P$ satisfies
\begin{align}
\label{eq:epipole_is_kernel}
F_P(e_{P_1}^{2},-)=F_P(-,e_{P_2}^{1})=0,
\end{align}
where
\begin{align}
\label{eq:definition_of_epipole}
e_{P_i}^{j}=P_i(p_j).
\end{align}
\end{proof}

\begin{remark} 
Recall that the entries in the camera matrix are real, this means that the fundamental form always has real coefficients. 
\end{remark}

\begin{definition}
\label{def:epipole}
The \emph{epipoles} $e_{P_i}^{j}$ are the image points we get by mapping the $j$-th camera center to the $i$-th image
\begin{align*}
e_{P_i}^{j}=P_i(p_j).
\end{align*}
\end{definition}

The fundamental form satisfies
\begin{align*}
F_P(e_{P_1}^{2},-)=F_P(-,e_{P_2}^{1})=0,
\end{align*}
in other words, it vanishes in either epipole. This means that the fundamental form is of rank 2 (also follows from $\im(\widetilde{\phi_P})$ being singular), so for each pair of cameras, we get a bilinear form of rank 2. The following result states that the converse is also true, i.e. that any bilinear form of rank 2 is the fundamental form for some pair of cameras. 

\begin{theorem}
\label{thr:fundforms_and_camera_pairs_are_1:1}
There is a $1:1$ correspondence between bilinear forms of rank two, and pairs of cameras $P_1,P_2$ (up to action by $\PGL(4)$)
\end{theorem}

\begin{proof}
By \cref{lem:fundamental_form}, the fundamental form of two cameras is a real bilinear form of rank 2. The converse follows from Theorem 9.13. in \cite{hartleyzisserman}.
\end{proof}

With these results, we can move on to classifying all the critical configurations for two views. We start with a special type of critical configuration:

\subsection{Trivial critical configurations}

\begin{definition}
\label{def_non_trivial_configuration}
A configuration $(P_1,P_2,X)$ is said to be a \emph{non-trivial critical configuration} if it has a conjugate configuration $(Q_1,Q_2,Y)$ satisfying
\begin{align*}
F_P\neq F_Q.
\end{align*}
\end{definition}

Critical configurations not satisfying this property exist, they are called trivial. If $(P_1,P_2,X)$ is a trivial critical configuration, then all its conjugates $(Q_1,Q_2,Y)$ have the same fundamental form as the cameras $P_1,P_2$. By \cref{thr:fundforms_and_camera_pairs_are_1:1} this means that, after a change of coordinates, $Q_1=P_1$ and $Q_2=P_2$. Since the cameras are the same, \cref{lem:joint-camera_map_is_isomorphism} tells us that the sets $X$ and $Y$ are equal, with the exception of any point lying on the line spanned by the two camera centers. It is a well-known fact that no number of cameras can differentiate between points lying on a line containing all the camera centers, hence the name \enquote{trivial}.

The focus of this paper is the non-trivial critical configurations. A classification of the trivial critical configurations for any number of views can be found in \cite{HK} Section 4, or in \cite{hartleyzisserman} Chapter 22.

\subsection{Critical configurations for two views}
\label{subsec:crit_for_two_views}
Let us consider a non-trivial critical configuration $(P_1,P_2,X)$. Since it is critical, there exists a conjugate configuration $(Q_1,Q_2,Y)$ giving the same images in $\pxp$. The two sets of cameras define two joint-camera maps $\widetilde{\phi_P}$ and $\widetilde{\phi_Q}$.
\begin{center}
\begin{tikzcd}[ampersand replacement=\&, column sep=small]
\widetilde{\p3} \arrow[rd, "\widetilde{\phi_P}"] \&   \& \widetilde{\p3} \arrow[ld, "\widetilde{\phi_Q}"'] \\
                          \& \pxp \&                           
\end{tikzcd}
\end{center}

Now, since the configuration is critical, we have that $\widetilde{\phi_P}(X)=\widetilde{\phi_Q}(Y)$. As such, the two sets $X$ and $Y$ both map (with their respective maps) into the intersection of the two multi-view varieties $\im(\widetilde{\phi_{P}})\cap \im(\widetilde{\phi_{Q}})$. Taking the preimage of this intersection under $\widetilde{\phi_{P}}$, we get a variety $S$ in $\widetilde{\p3}$ which needs to contain all the points in $X$. Moreover, if $(P_1,P_2,X)$ is maximal, then $X=S$. As such, classifying all non-trivial maximal critical configurations can be done by classifying all possible intersections between two multi-view varieties, and then examining what these intersections pull back to in $\widetilde{\p3}$. This is made even simpler by the fact that the pullback of $\im(\widetilde{\phi_{P}})\cap\im(\widetilde{\phi_{Q}})$ is just the variety we get by pulling back the fundamental form $F_Q$ (the fundamental form $F_P$ pulls back to the zero polynomial). 

The pullback of a bilinear form describes the strict (or proper) transform of a quadric surface\footnote{This is a surface over the complex numbers. The real points on this surface will generally also form a surface, but we will later see that there is one exception, namely when the surface is the union of two complex conjugate planes, and only their line of intersection is real}. It follows that $X$ and $Y$ lie on the strict transform of two quadric surfaces (quadrics) which we denote by $S_P$ and $S_Q$ respectively. Let $\widetilde{S_P}$ and $\widetilde{S_Q}$ denote their strict transforms. These quadrics are given by the following equations:
\begin{align}
\label{eq:SP_and_SQ}
\begin{split}
S_P(x)&=F_Q(P_1(x),P_2(x)),\\
S_Q(x)&=F_P(Q_1(x),Q_2(x)).
\end{split}
\end{align}
The two quadrics have the following properties:
\begin{lemma}[{Lemma 5.10 in \cite{HK}}]
\label{lem:properties_of_the_quadrics}
\begin{enumerate}
\item The quadric $S_P$ contains the camera centers $p_1$, $p_2$.
\item The quadric $S_P$ is ruled (contains real lines).
\end{enumerate}
\end{lemma}

\begin{proof}
The first item follows from the nature of the pullback, whereas the second is because we are working with forms of rank 2. Detailed proof is given in \cite{HK}.
\end{proof}

There are only four quadrics containing lines (up to choice of coordinates), these are illustrated in \cref{fig:ruled_quadrics}. 

\begin{figure}[]
\begin{center}
\includegraphics[width = 0.90\textwidth]{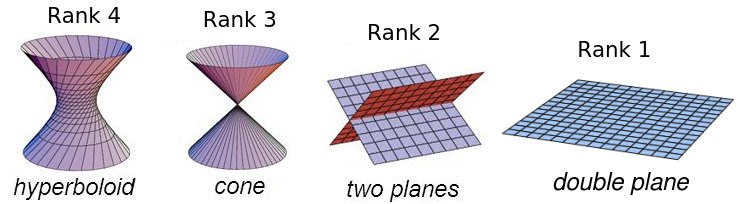}
\end{center}
\caption{All types of real ruled quadrics.}
\label{fig:ruled_quadrics}
\end{figure}

The discussion so far can be summarized as follows:
\begin{theorem}[{Lemma 5.10 in \cite{HK}}]
\label{thr:critical_configurations_lie_on_quadrics}
Let $(P_1,P_2,X)$ be a non-trivial critical configuration. Then there exists a ruled quadric $S_P$ passing through the camera centers $p_1$, $p_2$, whose strict transform contains the points $X$.
\end{theorem}

\cref{thr:critical_configurations_lie_on_quadrics} tells us that all non-trivial critical configurations have their points and camera centers lying on a ruled quadric. The converse, however, is not always true. For instance, we will soon see that cameras and camera centers all lying on a cone is not critical if the cone contains the line spanned by the two camera centers (see \cref{fig:critical_quadrics}). Let us instead give a partial converse:

\begin{lemma}
\label{lem:critical_if_pullback_of_rank2_form}
Let $(P_1,P_2,X)$ be a configuration of cameras and points such that $X$ is contained in the strict transform $\widetilde{S_P}$ of a ruled quadric $S_P$ that passes through both the camera centers. Then for each real bilinear form $F_Q$ of rank 2 such that:
\begin{align*}
S_P(x)=F_Q(P_1(x),P_2(x)),
\end{align*}
there exists a conjugate configuration to $(P_1,P_2,X)$.
\end{lemma}

\begin{proof}
Assume such a bilinear form $F_Q$ exists. By \cref{thr:fundforms_and_camera_pairs_are_1:1} there exists a pair of cameras $(Q_1,Q_2)$ such that $F_Q$ is their fundamental form. Since $X$ lies on $\widetilde{S_P}$, we have 
\begin{align*}
\widetilde{\phi_{P}}(X)\subset\im(\widetilde{\phi_Q}).
\end{align*}
Then for every point $x\in X$ we can find a point $y$ such that
\begin{align*}
\widetilde{\phi_P}(x)=\widetilde{\phi_Q}(y).
\end{align*}
Let $Y$ be the set of these points $y$. Then $(Q_1,Q_2,Y)$ is a conjugate to $(P_1,P_2,X)$. This can be repeated for each bilinear form of rank 2, giving unique conjugate configurations.
\end{proof}

The problem of determining which configurations are critical is now reduced to finding out which quadrics are the pullbacks of real bilinear forms of rank 2.

Let $\textbf{F}$ denote the space of bilinear forms on $\pxp$. Since all such forms can be represented by a $3\times3$ matrix (up to scaling) we have that $\textbf{F}$ is isomorphic to $\p8$. The fundamental form $F_P$ of the pair $(P_1,P_2)$ is an element in $\textbf{F}$. 

\begin{lemma}
\label{lem:correspndence_between_lines_and_quadrics}
There is a $1:1$ correspondence between the set of real quadrics in $\p3$ passing through $p_1,p_2$, and the set of real lines in $\textbf{F}$ passing through $F_P$.
\end{lemma}

\begin{proof}
Let $L\subset\textbf{F}$ be a line through $F_P$, and let $F_0\neq F_P$ be some point on $L$. Every point $F\in L$ can be written as $\alpha F_0+\beta F_P$ for some $[\alpha:\beta]\in\p1$. But then we have for $F\neq F_P$:
\begin{align*}
F(P_1(x),P_2(x))&=\alpha F_0(P_1(x),P_2(x))+\beta F_P(P_1(x),P_2(x)),\\
&=\alpha F_0(P_1(x),P_2(x))+\beta \cdot 0,\\
&=F_0(P_1(x),P_2(x)).
\end{align*} 
Hence, the equation $F(P_1(x),P_2(x))$ describes the same quadric for all points $F\neq F_P$ on $L$.

Next, a quadric $S$ passing through $p_1$ and $p_2$ is fixed by a set of 7 points $\Set{x_1,\ldots,x_7}$ on $S$ in generic position. Demanding that a bilinear form pulls back to a quadric passing through a specific point is one linear constraint in $\textbf{F}$. So with seven generic points, there is exactly one line $L$ through $F_P$ such that the forms on this line pull back to $S$.
\end{proof}

Using this $1:1$ correspondence and \cref{lem:critical_if_pullback_of_rank2_form}, the problem has been reduced to determining which quadrics correspond to lines in $\textbf{F}$ containing at least one viable form of rank 2. Let $\textbf{F}_2$ denote the Zariski closure of the rank 2 locus. Since $\textbf{F}_2$ is a hypersurface of degree 3, a generic line $L$ will contain two forms of rank 2 in addition to $F_P$. There are also other possibilities, listed in the tables below (the underlying computations for these tables can be found in the appendix.

We start with the cases where the line $L$ corresponding to $S_P$ has a finite number of intersections with the rank 2 locus $\textbf{F}_2$.

\begin{table}[h]
\begin{tabular}{@{}p{0.5\textwidth}p{0.45\textwidth}@{}}
\hline
\textbf{Intersection points}                                                                         & \textbf{$S_P$}                                                         \\ \hline
All three intersection points are distinct real points       & A smooth quadric, cameras not on a line                       \\ \hline
Two intersection points $F_P$ and $F_Q$, $L\cap\textbf{F}_2$ has multiplicity 2 at $F_Q$  & A cone, two cameras not on a line, neither camera on a vertex \\ \hline
Two intersection points $F_P$ and $F_Q$, $L\cap\textbf{F}_2$ has multiplicity 2 at $F_P$ & A smooth quadric, cameras lie on a line                       \\ \hline
\end{tabular}
\end{table}
\FloatBarrier
These are the cases where we have at least one real rank 2 form $F\in L$ different from $F_P$. There are, however, some cases where there are no viable forms:

\FloatBarrier
\begin{table}[h]
\begin{tabular}{@{}p{0.5\textwidth}p{0.45\textwidth}@{}}
\hline
\textbf{Intersection points}                                                                         & \textbf{$S_P$}                                                         \\ \hline
The two other intersection points are complex conjugates                               & A smooth non-ruled quadric                              \\ \hline
The two other intersection points are of rank 1 & Union of two planes, cameras in different planes             \\ \hline
The two intersection points are equal to $F_P$                        & A cone, both cameras on a line, neither camera at the vertex \\ \hline
\end{tabular}
\end{table}
\FloatBarrier
In the case where $L$ is contained in $\textbf{F}_2$, all the forms on $L$ are of rank 2 (with the possible exception of at most 2 that can be of rank 1). As such, rather than looking at where the intersections are, we look at the epipoles of the forms in $L$:
\FloatBarrier

\begin{table}[ht]
\begin{tabular}{@{}p{0.5\textwidth}p{0.45\textwidth}@{}}
\hline
\textbf{Epipoles}                                                                         & \textbf{$S_P$}                                                          \\ \hline
All forms have different epipoles                                                & Two planes, cameras lying in same plane                          \\ \hline
All forms share the same right epipole, the left epipoles trace a line\footnotemark  & Two planes, one camera on the intersection of the planes         \\ \hline
All forms share the same right epipole, the left epipoles trace a conic\footnotemark[\value{footnote}] & Cone, one camera at the vertex                                   \\ \hline
All forms share the same right and left epipole                                  & Two (possibly complex) planes, both cameras lying on the intersection of the planes \\ \hline
All forms share the same right and left epipole AND the two rank one forms on $L$ coincide & Double plane (as a set it is equal to a plane, but every point has multiplicity 2)                    \\ \hline
\end{tabular}
\end{table}
\FloatBarrier
\footnotetext{The statement also holds if we swap "left" and "right".}
With this, we have a classification of all maximal critical configurations for two views:

\begin{theorem}
\label{thr:critical_conf_for_two_views}
A configuration $(P_1,P_2,X)$ is non-trivially critical if and only if there exists a real quadric $S_P$ containing the camera centers $p_1,p_2$ and whose strict transform contains $X$, and $S_P$ is one of the quadrics in \cref{tab:configurations_and_their_conjugates}. (illustrated in \cref{fig:critical_quadrics})
\end{theorem}

\begin{table}[]
\begin{tabular}{@{}p{0.54\textwidth}p{0.22\textwidth}p{0.16\textwidth}@{}}
\toprule
\textbf{Quadric $S_P$}      & \textbf{Conjugate quadric} & \textbf{Conjugates}\\
\midrule
Smooth quadric, cameras not on a line       & Same    & 2 \\ \midrule
Smooth quadric, cameras on a line           & Cone, cameras not on a line & 1\\ \midrule
Cone, cameras not on a line & Smooth quadric, cameras on a line           & 1\\ \midrule
Cone, one camera at vertex, other one not & Same  & $\infty$ \\ \midrule
Two planes, cameras in the same plane               & Same   &$\infty$ \\ \midrule
Two planes, one camera on the singular line         & Same   &$\infty$ \\ \midrule
Two planes, cameras on the singular line       & Same   &$\infty$ \\ 
\midrule
A double plane, cameras in the plane     & Same   &$\infty$ \\
\bottomrule
\end{tabular}
\caption{A list of all possible non-trivial critical configurations and their conjugates, as well as the number of conjugates for each configuration.}
\label{tab:configurations_and_their_conjugates}
\end{table}

\begin{figure}[]
\begin{center}
\includegraphics[width = \textwidth]{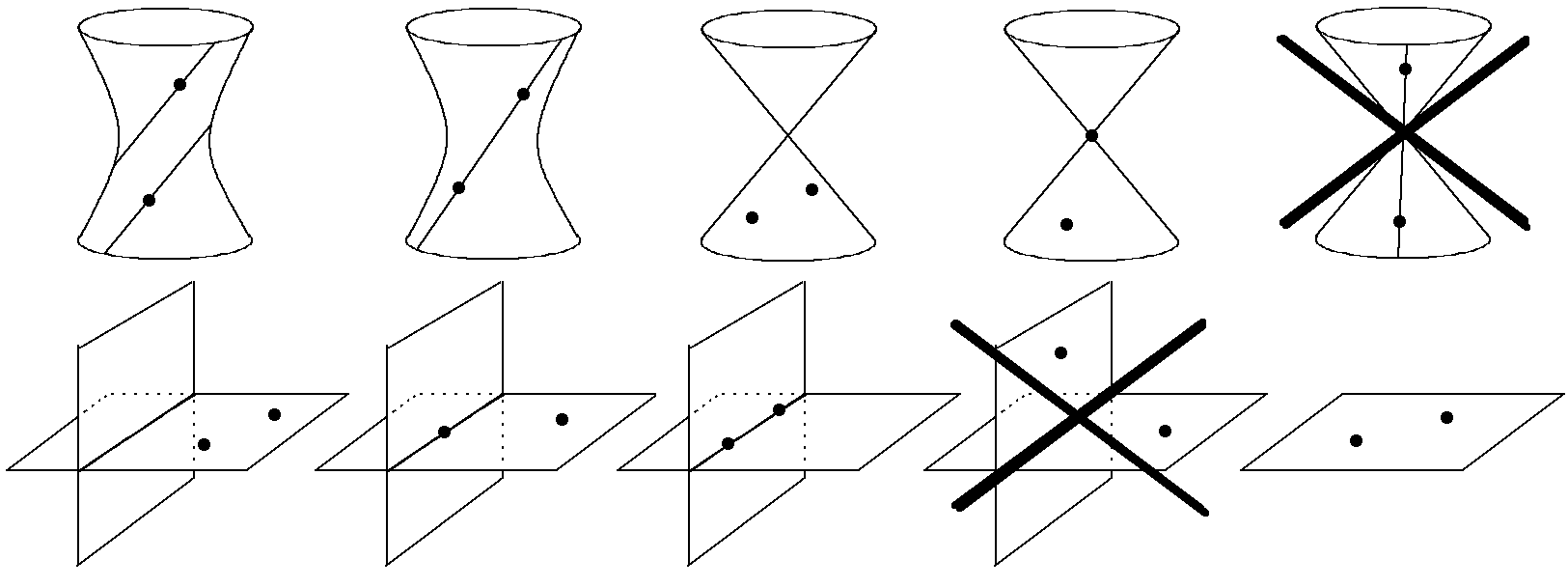}
\end{center}
\caption{Illustration of the blow-downs of the non-trivial critical configurations for two views. The two marked with crosses are not critical.}\label{fig:critical_quadrics}
\end{figure}

Recall that by \cref{def:equvalent_configurations}, we required two conjugate configurations to not be projectively equivalent. Yet by \cref{thr:critical_conf_for_two_views}, most critical configurations have conjugates that are of the same type. Now, while there is indeed some $A\in\PGL(4)$ taking any smooth quadric $S_P$ to any other smooth quadric $S_Q$, we will soon see that the map taking a point in $S_P$ to its conjugate on $S_Q$ certainly does not lie in $\PGL(4)$. In the final section, we give a description of the map taking a point to its conjugate, to make it clear that it is not a projective transformation.

\section{Maps between quadrics}
\label{sec:maps_between_quadrics}
\subsection{Epipolar lines}
Before we can describe the map taking a point to its conjugate, we need to point out a certain pair of lines on $S_P$. Given two pairs of cameras $(P_1,P_2)$ and $(Q_1,Q_2)$, let $\widetilde{S_P}$ be the pullback of $F_Q$ using $\widetilde{\phi_{P}}$, and define
\begin{align}
\label{eq:g1}
\widetilde{g_{P_1}^{2}}=P_1^{-1}(e_{Q_1}^{2})=P_1^{-1}(Q_1(q_2)),\\
\label{eq:g2}
\widetilde{g_{P_2}^{1}}=P_2^{-1}(e_{Q_2}^{1})=P_2^{-1}(Q_2(q_1)).
\end{align}
The blowdown of $\widetilde{g_{P_1}^{2}}$ and $\widetilde{g_{P_2}^{1}}$ are two lines on $S_P$, we denote them by $g_{P_i}^{j}$. Since they are the pullback of the epipoles from the other set of cameras, so we call them \emph{epipolar lines}.

Epipolar lines are key in understanding the relation between points on $\widetilde{S_P}$ and points on its conjugate $\widetilde{S_Q}$. They also play an important role in the study of critical configurations for more than 2 cameras, so let us give a brief analysis of these lines. 

\begin{restatable}[Lemma 5.10, Definition 5.11 in \cite{HK}]{lemma}{lemmaPermissible}
\label{lem:permissible_lines}
\begin{enumerate}
\item The line $\gp{i}{j}$ lies on $S_P$ and passes through $p_i$.
\item Any point lying on both $\gp{i}{j}$ and $\gp{j}{i}$ is a singular point on $S_P$.
\item Any point in the singular locus of $S_P$ that lies on one of the lines also lies on the other.
\item If $S_P$ is the union of two planes, $\gp{i}{j}$ and $\gp{j}{i}$ lie in the same plane.
\end{enumerate}
\end{restatable}
The first two properties are taken from Lemma 5.10 in \cite{HK}, the last two are neither stated nor proven in the paper. Nevertheless, the authors seem to have been aware of all four properties.
\begin{proof}
\begin{enumerate}
\item[1-2.]  See proof of Lemma 5.10 in \cite{HK}.
\item[3.] For ease of reading, we use matrix notation. As such, $S_P$, $F_Q$, and $P_i$ are represented by matrices of dimensions $4\times4$, $3\times3$ and $3\times4$ respectively. In particular, $S_P$ is represented by the \emph{symmetric} matrix $P_1^TF_QP_2+P_2^TF_Q^TP_1$.

If $x_0\in\gp{1}{2}$ lies in the singular locus of $S_P$, then we have
\begin{align*}
S_Px_0&=(P_1^{T}F_QP_2+P_2^{T}F_Q^{T}P_1)x_0,\\
&=P_1^{T}F_QP_2x_0+P_2^{T}F_Q^{T}e_{Q_1}^{2}),\\
&=P_1^{T}F_QP_2x_0.
\end{align*}
However, since $x_0$ lies in the singular locus of $S_P$, this expression is equal to zero. Since both the camera matrices are of full rank, the only way we can get zero is if $x_0=p_2$ or if $P_2x_0=e_{Q_2}^{1}$. In either case, it follows that $x_0$ lies on $\gp{2}{1}$ as well.
\item[4.] Assume there exist cameras $P_1,P_2,Q_1,Q_2$ such that $S_P$ is the union of two planes and the epipolar lines $g_{P_i}^{j}$ lie in different planes. When the quadric $S_P$ consists of two planes, one of the planes, which we denote by $\Pi$, will (by \cref{thr:critical_conf_for_two_views}) contain both camera centers. As such, the only way that the epipolar lines can lie in different planes is if one of the camera centers, say $p_2$, lies on the intersection of the two planes and the other does not (if both lie on the intersection, then by 3., the epipolar lines must both be equal to the intersection of the two planes). 

Recall that the quadric $S_P$ and its conjugate $S_Q$ (also two planes) are both pullbacks  of the surface $\im(\widetilde{\phi_P})\cap\im(\widetilde{\phi_Q})\subset\pxp$. The map $\widetilde{\phi_P}$ takes the plane $\Pi$ to the product of two lines in $\pxp$. The line in the first image passes through the epipole $e_{Q_1}^{2}$, whereas the line in the second image does not pass through the epipole $e_{Q_2}^{1}$ (this is because $\Pi$ contains one of the epipolar lines but not the other). The problem is now that neither of the planes on $S_Q$ can map to the product of these two lines, since any such plane would have to be both 
\begin{enumerate}
\item a plane passing through $q_2$ but not through $q_1$ (because in the second image, the line does not pass through $e_{Q_2}^{1}$) and
\item a plane passing through both $q_1$ and $q_2$ (because the plane maps to a line in both images),
\end{enumerate}
which gives us a contradiction. It follows that there are no $P_1,P_2,Q_1,Q_2$ such that $S_P$ is the union of two planes and the epipolar lines $g_{P_i}^{j}$ lie in different planes, so whenever $S_P$ is the union of two planes, the epipolar lines lie in the same plane. \qedhere
\end{enumerate}
\end{proof}

\begin{definition}
Let $S_P$ be a quadric surface and let $p_1,p_2$ be two distinct points on $S_P$. A pair of lines $g_{P_1}^{2},g_{P_2}^{1}$ is called \emph{permissible} if $g_{P_1}^{2},g_{P_2}^{1}$ satisfy the four conditions in \cref{lem:permissible_lines}.
\end{definition}

\begin{proposition}
\label{prop:conjugates_and_permissible_lines}
Let $P_1,P_2$ be two cameras, and let $S_P$ be a quadric passing through their camera centers. The configuration $(P_1,P_2,\widetilde{S_P})$ is critical if and only if $S_P$ contains a permissible pair of lines. Furthermore, if $p_1,p_2$ do not both lie in the singular locus of $S_P$, there is a 1:1 correspondence between permissible pairs of lines and configurations conjugate to $(P_1,P_2,\widetilde{S_P})$.
\end{proposition}
\begin{proof}
The first part can be proven by comparing the quadrics in  \cref{tab:configurations_and_their_conjugates} to the set of quadrics containing the required lines, and noting that they are the same. We leave this to the reader. 

The second part is immediate for the three cases where there is a finite number of conjugates since the number of conjugates is equal to the number of pairs of epipolar lines and each conjugate comes with its unique choice of lines. For the remaining quadrics $S_P$ (cone and two planes), there exists a pencil of fundamental forms, where each $F_Q\neq F_P$ satisfies
\begin{align*}
S_P(x)=F_Q(P_1(x),P_2(x)).
\end{align*}
By the tables on pages 11 and 12, we get that each form in the pencil yields a different pair of lines as long as the camera centers do not both lie in the singular locus of $S_P$. For $F_Q\neq F_P$, the lines are permissible, while for $F_P$, the lines coincide (not permissible). Note also that on each of these quadrics, the set of permissible lines forms a set whose Zariski closure is also a pencil (the pair where the two lines coincide lie in the closure, but the pair itself is not permissible). We now have a map from the pencil of fundamental forms to the pencil of permissible lines. Since the map is not constant, it needs to be surjective, and since no two fundamental forms in the pencil share the same left and right kernel, it is also injective. Hence there is a 1:1 correspondence between permissible pairs of lines and configurations conjugate to $(P_1,P_2,S_P)$.
\end{proof}

\subsection{Maps between quadrics}
Let us now have a closer look at the map $\psi$ taking a point to its conjugate. Given two pairs of cameras $P_1,P_2$ and $Q_1,Q_2$ with camera centers $p_1,p_2$ and $q_1,q_2$ respectively, let the quadrics $S_P$ and $S_Q$ and the epipolar lines $g$ be defined as before (Equations (\ref{eq:SP_and_SQ}), (\ref{eq:g1}), and (\ref{eq:g2})). Let 
\begin{align*}
\pi_P\from\widetilde{\p3_P}\to\p3
\end{align*}
be the blow-up of $\p3$ in the two camera centers $p_1,p_2$ and let $\widetilde{S_P}$ denote the strict transform of $S_P$, and similarly for $Q$. Define the incidence variety:
\begin{align*}
I=\Set{(x,y)\in \widetilde{S_P}\times \widetilde{S_Q}\mid \widetilde{\phi_P}(x)=\widetilde{\phi_Q}(y)},
\end{align*}
where $\phi$ is the joint-camera map. If we can understand $I$, we will know the exact relation between points on one quadric and the other. We have the following the commutative diagram:

\begin{center}
\begin{tikzcd}
& I \arrow[ld, "\pi_1"'] \arrow[rd, "\pi_2"] &                                                                                                                & \subset\widetilde{\p3}\times\widetilde{\p3} \\
\widetilde{S_P} \arrow[d, "\pi_P"'] \arrow[rdd, "\widetilde{\phi_P}"] \arrow[rr, "\psi_P"', dashed, shift right] &                                            & \widetilde{S_Q} \arrow[d, "\pi_Q"] \arrow[ldd, "\widetilde{\phi_Q}"'] \arrow[ll, "\psi_Q"', dashed, shift right] & \subset\widetilde{\p3}                      \\
S_P \arrow[rd, "\phi_P"', dashed]                                                                              &                                            & S_Q \arrow[ld, "\phi_Q", dashed]                                                                               & \subset\p3                                  \\
                                                                                                               & \pxp                                       &                                                                                                                &                                            
\end{tikzcd}
\end{center}

We study $I$ by studying the fibers of the projection map $\pi_1$. For any point $x\in\widetilde{S_P}$, we have
\begin{align*}
\pi_1^{-1}(x)=l_1\cap l_2,
\end{align*}
where $l_i$ is the line $Q_i^{-1}(P_i(x))$. We will rarely refer to this formula explicitly, but it is the foundation for the analysis of the fibers.

\begin{lemma}
\label{lem:singular_quadric_one_side_means_line_on_the_other}
\begin{enumerate}
\item If $\widetilde{S_P}$ is smooth, the map $\pi_1$ is an isomorphism. 
\item The map $\psi_P$ taking a point $x\in\widetilde{S_P}$ to its conjugate is a birational map, defined everywhere except the intersection $\widetilde{\gp{1}{2}}\cap\widetilde{\gp{2}{1}}$.
\item The quadric $S_P$ is singular if and only if the quadric $S_Q$ contains the line spanned by the camera centers $q_1,q_2$.
\item For any point $x\in\widetilde{\gp{1}{2}}$ or $\widetilde{\gp{2}{1}}$ (but not both), the conjugate $\psi_P(x)$ is a point on the exceptional divisor we get when blowing up $q_2$ or $q_1$ respectively. In particular, on the blow-down $S_P\subset\p3$, the conjugate to any point $x\in\gp{1}{2}$ or $x\in\gp{2}{1}$ (apart from the camera centers themselves) is the camera center $q_2$ or $q_1$ respectively.

\end{enumerate} 
\end{lemma}
\begin{proof}
\begin{enumerate}
\item As long as $x$ does not lie in $\widetilde{\gp{1}{2}}\cap\widetilde{\gp{2}{1}}$, the two lines $l_i$ do not coincide, so the fiber $\pi_1^{-1}(x)$ consists of a single point. By \cref{lem:permissible_lines}, any point lying on this intersection is a singular point on $S_P$, so if $S_P$ is smooth, then $\widetilde{\gp{1}{2}}$ and $\widetilde{\gp{2}{1}}$ do not intersect. Then every fiber is a singleton, meaning that $\pi_1$ is injective. Furthermore, for each point $x\in\widetilde{S_P}$ there is at least one point $y\in\widetilde{S_Q}$ such that $\widetilde{\phi_P}(x)=\widetilde{\phi_Q}(y)$, so $\pi_1$ is surjective as well.
\item As mentioned above, $\pi_1^{-1}(x)$ is a singleton if $x$ does not lie in $\widetilde{\gp{1}{2}}\cap\widetilde{\gp{2}{1}}$. If it does, on the other hand, the fiber is a line $L\in I$ such that $\pi_2(L)$ is the strict transform of the line spanned by the camera centers $q_1,q_2$, so in this case the conjugate is not unique. Since the same is true for the map taking a point $y\in\widetilde{S_Q}$ to its conjugate, the map is birational.

\item If $S_P$ is singular, the epipolar line $\gp{1}{2}$ passes through some singular point $x\in\widetilde{S_P}$ (as do all other lines, see \cref{fig:ruled_quadrics}). By \cref{lem:permissible_lines}, the other epipolar line $\gp{2}{1}$ passes through the same point. Since the two epipolar lines intersect, the quadric $\widetilde{S_Q}$ must contain the line spanned by the camera centers.

Conversely, if $S_Q$ contains the line spanned by the camera centers, the conjugate to any point on the strict transform of this line is the intersection $\widetilde{\gp{1}{2}}\cap\widetilde{\gp{2}{1}}$ on $\widetilde{S_P}$. By \cref{lem:permissible_lines}, the blow-down of these points are always singular points on $S_P$. 

\item Let $x$ be a point lying on the epipolar line $\widetilde{\gp{1}{2}}$, but not on $\widetilde{\gp{2}{1}}$. We have $P_1(x)=e_{Q_1}^2$ and $P_2(x)\neq e_{Q_2}^1$. The conjugate to $x$ is then
\begin{align*}
\psi(x)=l_1\cap l_2=Q_1^{-1}(P_1(x))\cap Q_2^{-1}(P_2)(x).
\end{align*}
The line $l_1$ is here the line spanned by the camera centers $q_1,q_2$ , while $l_2$ is a different line, passing through $q_2$. Their intersection must then be the camera center $q_2$. The proof for the other epipolar line is the same.
\qedhere
\end{enumerate}
\end{proof}

\begin{remark}
\label{rem:method_for_computing}
When the two pairs cameras $P_1,P_2$ and $Q_1,Q_2$ are known, one can explicitly compute the map $\psi_P$ as a rational function. For some representatives of a general point $X\in S_P$ and its conjugate $\psi_P(X)\in S_Q$, we have
\begin{align*}
\underbrace{\begin{pmatrix}
Q_1&P_1X&0\\
Q_2&0&P_2X
\end{pmatrix}}_{A}\cdot\begin{pmatrix}
\psi_Q(X)\\1\\1
\end{pmatrix}=\begin{pmatrix}
P_1&Q_1\psi_Q(X)&0\\
P_2&0&Q_2\psi_Q(X)
\end{pmatrix}\cdot\begin{pmatrix}
X\\1\\1
\end{pmatrix}=0.
\end{align*}
Using the adjugate of $A$, we can get an expression for $\psi_Q$ as a rational function in the homogeneous coordinates of $X$.
\end{remark}

Let us now give two results on how $\phi_P$ acts on the curves on $S_P$:

\begin{definition}
\label{def:type_of_curve}
Let $S_P$ be a smooth quadric or cone and let $C_P$ be a curve on $S_P$. We say that $C_P$ is of type $(a,b,c_1,c_2)$, where $c_i$ is the multiplicity of $C_P$ in the camera center $p_i$ and where $a$ and $b$ is:
\begin{itemize}
\item If $S_P$ is smooth, $a$ is the number of times $C_P$ intersects a generic line in the same family as the epipolar lines $g_{P_i}^{j}$, and $b$ is the number of times it intersects the lines in the other family, (meaning $(a,b)$ is the bidegree of $C_P$).
\item If $S_P$ is a cone, $a$ is the number of times $C_P$ intersects each line outside of the vertex, and $b$ the number of times it intersects each line.
\end{itemize}
\end{definition}

\begin{proposition}[{\cite[Lemma 8.32]{HK}}]
\label{prop:bidegree_on_P_->_bidegree_on_Q}
Let $S_P$ be a smooth quadric or a cone, and let $C_P\subset S_P$ be a curve of type $(a,b,c_1,c_2)$, such that $C_P$ does not contain either of the epipolar lines. Then the conjugate curve $C_Q\subset S_Q$ is of type $(a,a+b-c_1-c_2,a-c_2,a-c_1)$.
\end{proposition}

An illustration of how $\psi_P$ acts on the lines on a quadric/cone can be found in Figures \ref{fig:conjugates1} and \ref{fig:conjugates2}.

\begin{definition}
Let $S_P$ be two planes with two camera centers not both lying on the intersection and let $C_P$ be a curve on $S_P$. We say that $C_P$ is of type $(a,b,c_0,c_1,c_2)$, where $a$ is the degree of the curve in the plane with the epipolar lines, $b$ the degree of the curve in the other plane, $c_0$ is the multiplicity of $C_P$ in the intersection of the epipolar lines $g_{P_i}^{j}$ and $c_1,c_2$ is the multiplicity in the camera centers $p_1,p_2$ respectively. 
\end{definition}

\begin{proposition}
\label{prop:bidegree_on_P_->_bidegree_on_Q_planes}
Let $S_P$ be two planes with two camera centers not both lying on the intersection, and let $C_P\subset S_P$ be a curve of type $(a,b,c_0,c_1,c_2)$ such that $C_P$ does not contain either of the epipolar lines or the line spanned by the camera centers. Then the conjugate curve $C_Q\subset S_Q$ is of type $(2a-c_0-c_1-c_2,b,a-c_1-c_2,a-c_0-c_2,a-c_0-c_1)$.
\end{proposition}

The goal of the remainder of this section will be to give a rigorous proof of Lemma 8.32 in \cite{HK} and to generalize this result to also hold for singular quadrics (that is, to prove \cref{prop:bidegree_on_P_->_bidegree_on_Q,prop:bidegree_on_P_->_bidegree_on_Q_planes}). While the method described in \cref{rem:method_for_computing} gives us an explicit expression for $\psi_P$, using this to prove \cref{prop:bidegree_on_P_->_bidegree_on_Q,prop:bidegree_on_P_->_bidegree_on_Q_planes} is somewhat difficult. We instead switch our approach to a more classical one, using intersection theory. The map $\psi_P$ taking a point on $\widetilde{S_P}$ to its conjugate can be described by describing the pullback of the hyperplane sections of $\widetilde{S_Q}$. For such hyperplane sections $H$, the pullbacks $\psi_P^{-1}(H)$ will be curves on $\widetilde{S_P}$. Finding the class of $\psi_P^{-1}(H)$ in the Chow ring will give us the information we need about $\psi_P$. We refer the reader to \cite{eisenbud_harris_2016} (chapter 2 in particular) for the basics on intersection theory used in the rest of the section.

By \cref{lem:singular_quadric_one_side_means_line_on_the_other}, the map $\psi$ is defined everywhere except the intersection of the two epipolar lines, so if we assume that they do not intersect, we get the diagram below:
\begin{center}
\begin{tikzcd}
\widetilde{\p3_P} \supset \widetilde{S_P} \arrow[d, "\pi_P"', shift left=5] \arrow[r, "\psi_P"', shift right] & \widetilde{S_Q}\subset \widetilde{\p3_Q} \arrow[d, "\pi_Q", shift right=5] \arrow[l, "\psi_Q"', shift right] \\
\p3 \supset S_P \arrow[r, dashed, shift right]                                                                & S_Q\subset \p3 \arrow[l, dashed, shift right]                                                               
\end{tikzcd}
\end{center}
If either pair of epipolar lines DO intersect, however (this happens whenever the quadrics are not smooth), then the map $\psi$ is not a morphism, but rather a rational map as it is not defined on the intersection. This can be mended by first blowing up the intersection of the epipolar lines, let
\begin{align*}
\pi_{g_P}\from\widetilde{\p3}\to\p3
\end{align*}
denote this blow-up, and let $\pi_P$, like before, be the blow-up in the two camera centers $p_i$. Let $\overline{S_P}$ be the strict transform of $S_P$ after the first blow-up, and let $\widetilde{S_P}$ be the strict transform of $\overline{S_P}$ after the second. And similarly for $Q$. We then get the following diagram:
\begin{center}
\begin{tikzcd}
\widetilde{\p3_P} \supset \widetilde{S_P} \arrow[r, "\psi_P"', shift right] \arrow[d, "\pi_P"', shift left=5] & \widetilde{S_Q}\subset\widetilde{\p3_Q} \arrow[l, "\psi_Q"', shift right] \arrow[d, "\pi_Q", shift right=5] \\
\widetilde{\p3}\supset\overline{S_P} \arrow[r, dashed, shift right] \arrow[d, "\pi_{g_P}"', shift left=5]     & \overline{S_Q}\subset\widetilde{\p3} \arrow[l, dashed, shift right] \arrow[d, "\pi_{g_Q}", shift right=5]   \\
\p3 \supset S_P \arrow[r, dashed, shift right]                                                                & S_Q\subset\p3 \arrow[l, dashed, shift right]                                                               
\end{tikzcd}
\end{center}

We can now cover the different cases one by one, but first, let us give a result which will be helpful to determine the intersection multiplicities of curves on $\widetilde{S_P}$:

\begin{proposition}[{\cite[Proposition 2.19]{eisenbud_harris_2016}}]
\label{Prop:eisenbud}
Let $S$ be a smooth projective surface and $\pi\from\widetilde{S}\to S$ the blow-up of $S$ at a point $p$; let $E\in A^1(S)$ be the class of the exceptional divisor.
\begin{enumerate}
\item $A(\widetilde{S})=A(S)\oplus\mathbb{Z}E$ as abelian groups.
\item $\pi^{*}\alpha\cdot\pi^{*}\beta=\pi^{*}(\alpha\beta)$ for any $\alpha,\beta\in A^{1}(S)$.
\item $E\cdot\pi^{*}\alpha=0$ for any $\alpha\in A^{1}(S)$.
\item $E^2=-1$.
\end{enumerate}
\end{proposition}
\subsubsection*{Smooth quadric}
We start by describing the map $\psi_P$ in the case where the quadric $S_P$ is smooth, i.e. a hyperboloid of one sheet. Let $L_2$ be the class of the total transform of a line in the same family as $g_{P_i}^{j}$, and let $L_1$ be the class of the total transform of a line in the other family. Let $E_i$ be the class of the exceptional divisor we get when blowing up $p_i$. 

On $S_P$, two lines from the same family do not intersect, whereas two lines in different families intersect once. By \cref{Prop:eisenbud} (2) this is also the case on $\widetilde{S_P}$, by (3) $L_i$ will not intersect $E_i$, and by (4), $E_i^2=-1$. We get the following intersection multiplicity table:

\begin{table}[h]
\begin{center}
\begin{tabular}{@{}|c|llll@{}}
\toprule
$\cdot$ & $L_1$ & $L_2$ & $E_1$ & $E_2$ \\ \midrule
$L_1$   & 0   & 1   & 0   & 0   \\
$L_2$   & 1   & 0   & 0   & 0   \\
$E_1$   & 0   & 0   & -1  & 0   \\
$E_2$   & 0   & 0   & 0   & -1  \\ \bottomrule
\end{tabular}
\end{center}
\end{table}

Next, let $H_Q$ be the pullback of a hyperplane section from the conjugate configuration $\widetilde{S_Q}$. We want to find the class of $H_Q$ on $\widetilde{S_P}$. 

First, let us consider what the map $\psi_P$ does to lines in each family. Let $l$ be a generic curve in $L_2$, in each image, $P_i(l)$ is a line, furthermore, since $l$ does not intersect $\widetilde{g_{P_i}^{j}}$, $P_i(l)$ is a line not passing through the epipole $e_{P_i}^{j}\in\p2$. It follows that $\psi(l)$ is a curve appearing as the intersection of $\widetilde{S_Q}$ and two planes (the preimages $Q_i^{-1}(P_i(l))$), each passing through exactly one camera center. In other words: $\psi(l)$ is a line, which means it intersects a generic hyperplane once. For a generic curve $l'$ in $L_1$, $P_i(l')$ is a line passing through the epipole, hence $\psi(l')$ is a curve appearing as the intersection of $\widetilde{S_Q}$ and two planes, both passing through the camera centers, in other words, a conic curve. This means $\psi(l')$ will intersect a generic hyperplane twice. 

Furthermore, by \cref{lem:singular_quadric_one_side_means_line_on_the_other} (4), the lines $g_{P_i}^{j}$, belonging to the class $L_2-E_i$, map to camera centers on $S_Q$. These will not intersect a generic hyperplane. Finally, $H_Q$ has self-intersection 2. This gives us the following equations: 
\begin{align*}
H_QL_2=1, \quad H_QL_1=2, \quad H_Q(L_2-E_i)=0, \quad H_Q^{2}=2.
\end{align*}
Let $H_Q=\alpha_1 L_1+\alpha_2 L_2+\alpha_3 E_1+ \alpha_4 E_2$. Using the intersection multiplicity table, we solve equations above for $\alpha_i$ and get $H_Q=L_1+2L_2-E_1-E_2$. In other words, the hyperplane sections on $\widetilde{S_Q}$ pull back to curves of bidegree $(1,2)$ passing through both camera centers.

\subsubsection*{Cone}
Next, let $S_P$ be a cone. Since the two epipolar lines intersect in the vertex, we first blow up $\p3$ in the vertex of the cone and then in the camera centers (see diagram above). Let $L$ be the class of the total transform (with the second blow-up) of the strict transform (with the first blow-up) of a line on $S_P$, let $E_0$ be the class of the exceptional divisor which is the blow-up of the vertex and let $E_i$ be the class of the exceptional divisor we get when blowing up $p_i$.

Since the cone is singular, we can not apply \cref{Prop:eisenbud} to the first blow-up, we need a separate argument. Blowing up the vertex of the cone, and taking the strict transform of the lines from $S_P$, we get curves $L$ which no longer intersect one another, so $L^2=0$, they do, however, intersect the exceptional divisor $E_0$, so $LE_0=1$. Moreover, a hyperplane section is of class $H=2L+E_0$, blowing up the cone, the resulting surface still has degree 2, so $H^2=2$, it follows that $E_0^2=-2$. When we next blow up the two camera centers we are blowing up a smooth surface, so \cref{Prop:eisenbud} applies, we get the following intersection multiplicity table:

\begin{table}[h]
\begin{center}
\begin{tabular}{@{}|c|llll@{}}
\toprule
$\cdot$ & $L$ & $E_0$ & $E_1$ & $E_2$ \\ \midrule
$L$     & 0   & 1   & 0   & 0   \\
$E_0$   & 1   & -2   & 0   & 0   \\
$E_1$   & 0   & 0   & -1  & 0   \\
$E_2$   & 0   & 0   & 0   & -1  \\ \bottomrule
\end{tabular}
\end{center}
\end{table}

Again, let $H_Q$ be the pullback of a hyperplane from the conjugate configuration $\widetilde{S_Q}$. By arguments similar to the smooth quadric case we get the following equations:
\begin{align*}
H_QL=1, \quad H_Q(L-E_i)=0, \quad H_Q^{2}=2.
\end{align*}
Furthermore, if neither camera center on $S_P$ is on the vertex, the conjugate quadric $S_Q$ is a smooth quadric with both cameras lying on the same line, the preimage of this line (under $\psi_P$) is $E_0$, hence $H_QE_0=1$. On the other hand, if one of the cameras, say $p_1$, lies on the vertex, $S_Q$ is a cone and the strict transform $E_0-E_1$ collapses to the vertex on $S_Q$, meaning $H_Q(E_0-E_1)=0$, so $H_QE_0=1$ in this case also.

It follows that $H_Q=3L+E_0-E_1-E_2$ in both cases.

\subsubsection*{Two planes, at most one camera center on the intersection}
Let $S_P$ be the union of two planes with at most one camera center lying on the intersection. In this case, $S_Q$ is reducible as well and $\psi_P$ takes the plane with the camera centers $p_i$ to the plane with the camera centers $q_i$ and the plane with no camera centers to the plane with no camera centers. As such we can consider the planes separately. The map $\psi_P$ restricted to the plane with no camera centers is an isomorphism, so we need only consider how $\psi$ acts on the plane with the camera centers. 

Since the two epipolar lines $g_{P_i}^{j}$ intersect in a point, we blow up this point, and then the camera centers to get a morphism $\psi_P$. Let $L$ be the class of the total transform of a line in the plane containing the camera centers, let $E_i$ be the class of the exceptional divisor we get when blowing up $p_i$, and let $E_0$ be the class of the exceptional divisor we get when blowing up the intersection of the epipolar lines. The plane is smooth, so by \cref{Prop:eisenbud} we get the following intersection multiplicity table:

\begin{table}[h]
\begin{center}
\begin{tabular}{@{}|c|llll@{}}
\toprule
$\cdot$ & $L$ & $E_0$ & $E_1$ & $E_2$ \\ \midrule
$L$     & 1   & 0   & 0   & 0   \\
$E_0$   & 0   & -1  & 0   & 0   \\
$E_1$   & 0   & 0   & -1  & 0   \\
$E_2$   & 0   & 0   & 0   & -1  \\ \bottomrule
\end{tabular}
\end{center}
\end{table}

Again, let $H_Q$ be the pullback of a hyperplane from the conjugate configuration $\widetilde{S_Q}$. The strict transform of the line spanned by the two camera centers on $S_P$ is mapped to the intersection of the epipolar lines on $S_Q$, similarly, the two epipolar lines map to camera centers. There are three corresponding lines on $S_Q$ whose strict transform maps to points on $S_P$. Lastly, since we are working with a single plane, $H_Q$ has self-intersection 1 this time. This gives us the equations:
\begin{align*}
&H_Q(L-E_1-E_2)=0, \quad\quad &&H_Q(L-E_0-E_1)=0, \quad\quad H_Q(L-E_0-E_2)=0,\\
&H_QE_i=1, \quad\quad &&H_Q^{2}=1.
\end{align*}
It follows that $H_Q=2L-E_0-E_1-E_2$.

\subsubsection*{One or two planes, both cameras in singular locus}
Finally, there is the case where $S_P$ (and hence $S_Q$) is either a double plane or two planes with both cameras lying on their intersection. In this case, the map $\psi_P$ is not defined on the line spanned by the two camera centers $p_1,p_2$ as any point on this line is conjugate to any point on the line spanned by $q_1,q_2$. Outside this line, $\psi_P$ acts simply as a linear transformation on each of the two (or one) planes. Note however that in the case of two planes the two linear transformations need not coincide on the intersection, for instance, two intersecting lines, one from each plane, might be taken to two disjoint lines. 
\begin{remark}
In the case where $S_P$ is a double plane, $\psi_P$ acts as a linear transformation on $S_P$ outside the line  spanned by the camera centers. As such, it is in many ways similar to the trivial critical configurations, although it is, by \cref{def_non_trivial_configuration}, non-trivial.
\end{remark}

With all these results in place, we can now prove \cref{prop:bidegree_on_P_->_bidegree_on_Q,prop:bidegree_on_P_->_bidegree_on_Q_planes}. Recall that \cref{def:type_of_curve} defines the type of a curve to be $(a,b,c_1,c_2)$ where $c_1$ and $c_2$ is the multiplicity in each camera center, and where $a$ and $b$ is the number of times it intersects the epipolar lines and the other lines respectively (or in the case of a cone, the number of times it intersects a generic line outside the vertex, and the number of times it does so in total respectively). In other words, a curve of type $(a,b,c_1,c_2)$ belongs to the class $aL_1+bL_2-c_1E_1-c_2E_2$ on a smooth quadric, whereas on the cone, it belongs to the class $(a+b)L+aE_0-c_1E_1-c_2E_2$.

\begin{proof} [Proof of \cref{prop:bidegree_on_P_->_bidegree_on_Q}]
Let $C_P\subset S_P$ be a curve of type $(a,b,c_1,c_2)$ and let $C_Q\subset S_Q$ be of type $(a',b',c_1',c_2')$. By \cref{lem:singular_quadric_one_side_means_line_on_the_other} (4), the pullback of the camera center $q_2\in S_Q$ is the epipolar line $g_{P_1}^{2}$, which belongs to the class $L_2-E_1$. $C_P$ intersects this line $a-c_1$ times. It follows that $c_2'=a-c_1$, similarly, $c_1'=a-c_2$.

Assume now that $S_P$ is smooth. A hyperplane section on $S_Q$ is of type $(1,1,0,0)$, we have shown above that the pullback of such a section is of type $(1,2,1,1)$ on $S_P$. Certain hyperplane sections of $S_Q$ are reducible, reducing into two components of type $(1,0,0,0)$ and $(0,1,0,0)$, with each component moving in a pencil. Their pullback should do the same. The only way for a $(1,2,1,1)$-curve to reduce into two such components is as $(0,1,0,0)$ and $(1,1,1,1)$, since bidegreee $(0,2)$ would mean it is reducible, and the bidegree $(0,1)$-component passing through a camera center would mean it can not move. This means that the pullback of the $(1,0,0,0)$ component is of type $(1,1,1,1)$ (since this is the component intersecting the epipolar lines, it should correspond to the component through the camera centers on the other side), similarly, the pullback of the $(0,1,0,0)$-component is of type $(0,1,0,0)$. It follows that $a'=a$ and $b'=a+b-c_1-c_2$, meaning $C_Q$ is of type $(a,a+b-c_1-c_2,a-c_2,a-c_1)$.

The proof in the case where $S_P$ is a cone follows a similar pattern: in this case, a hyperplane section $(1,1,0,0)$ on $S_Q$ pulls back to a curve in the class $3L+E_0-E_1-E_2$, that is, a curve of type $(1,2,1,1)$. Regardless of whether $S_Q$ is a cone or smooth quadric, there are (like in the previous case) hyperplane sections of $S_Q$ which reduce to two irreducible components, each free to move in a pencil. Similarly, their pullback must be reducible into two such components. The only way a curve in $3L+E_0-E_1-E_2$ reduces in such a way is as $L$ and $2L+E_0-E_1-E_2$ since $3L$ and $2L$ would both be further reducible, and the curve in class $L$ passing through a camera center would prevent it from moving. $L$ is of type $(0,1,0,0)$ and $2L+E_0-E_1-E_2$ is of type $(1,1,1,1)$. Following the same argument as in the smooth quadric case, $C_Q$ is of type $(a,a+b-c_1-c_2,a-c_2,a-c_1)$ in this case also.
\end{proof}

\begin{proof}[Proof of \cref{prop:bidegree_on_P_->_bidegree_on_Q_planes}]
Let $C_P\subset S_P$ be a curve of type $(a,b,c_0,c_1,c_2)$ and let $C_Q\subset S_Q$ be of type $(a',b',c_0',c_1',c_2')$. By \cref{lem:singular_quadric_one_side_means_line_on_the_other} (4) the pullback of the camera center $q_2\in S_Q$ is the epipolar line $g_{P_1}^{2}$, which belongs to the class $L-E_0-E_1$. $C_P$ intersects this line $a-c_0-c_1$ times. It follows that $c_2'=a-c_0-c_1$, similarly, we get $c_1'=a-c_0-c_2$. Moreover, by \cref{lem:singular_quadric_one_side_means_line_on_the_other} (3) the pullback of the intersection of the epipolar lines, is the lines spanned by the two camera centers (class $L-E_1-E_2$), it follows that $c_0'=a-c_1-c_2$.

A generic line in the plane containing the camera centers on $S_Q$ pulls back to a curve in the class $2L-E_0-E_1-E_2$ on $S_P$, $C_P$ intersects this $2a-c_0-c_1-c_2$ times, it follows that $a'=2a-c_0-c_1-c_2$. Lastly, $\phi_P$ is an isomorhpism on the final plane, so $b'=b$. This means that $C_Q$ is of type $(2a-c_0-c_1-c_2,b,a-c_1-c_2,a-c_0-c_2,a-c_0-c_1)$.
\end{proof}

\begin{figure}[]
\begin{center}
\includegraphics[width = 0.8\textwidth]{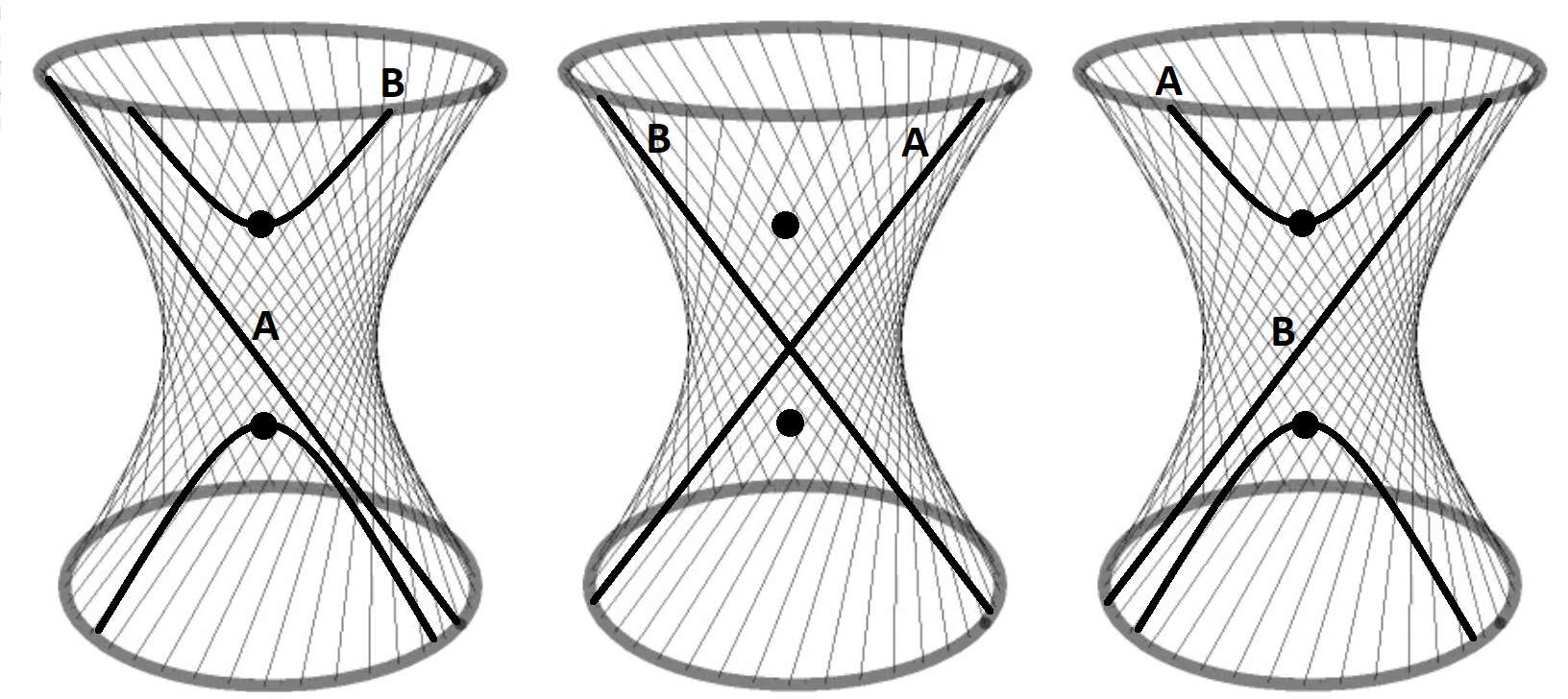}
\end{center}
\caption{Illustration of the map taking a point to its conjugate. The smooth quadric has two conjugates (original in the middle), the left is the one we get if we take the lines in the "A" family to be the epipolar lines, whereas the one on the right is the one we get if we choose "B". In both cases the lines in the family with the epipolar lines are preserved, whereas the lines in the other family are mapped to conic curves.}
\label{fig:conjugates1}
\end{figure}

\begin{figure}[]
\begin{center}
\includegraphics[width = 0.60\textwidth]{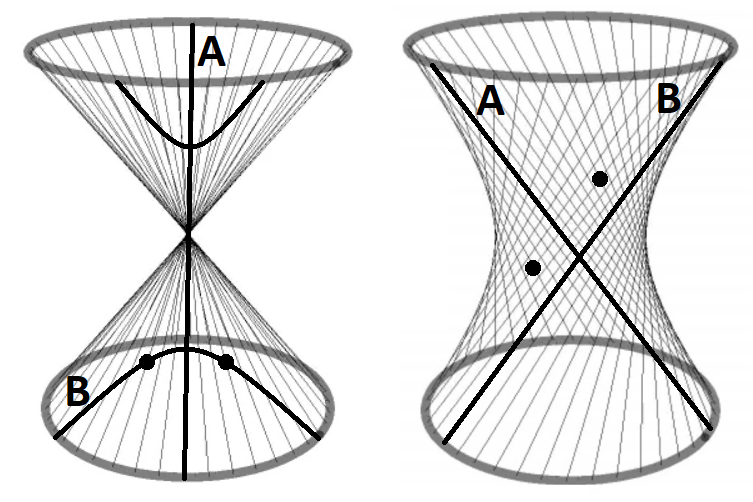}
\end{center}
\caption{Illustration of the map taking a point to its conjugate. The line spanned by the camera centers (right) maps to the vertex on the cone. The other lines in this family map to conics. Lines in the other family are preserved.}
\label{fig:conjugates2}
\end{figure}

\section*{Acknowledgments}
I would like to thank my two supervisors, Kristian Ranestad and Kathlén Kohn, for their help and guidance, for providing me with useful insights, and for their belief in my work. I would also like to thank the anonymous referees for providing helpful suggestions, and in particular for pointing out the method described in \cref{rem:method_for_computing}. Lastly, I thank Erin Connelly for pointing out an error in the table on page 11 and in the supplementary material, this error made it into the published paper, but is not present in this version. This work was supported by the Norwegian National Security Authority. 

\appendix
\section*{Appendix}

\subsection*{Proof behind tables in \cref{subsec:crit_for_two_views}}
\label{app:arguments_for_tables}
Let $P_1,P_2$ be two cameras with camera centers $p_1,p_2$, and fundamental form $F_P$. Let $\textbf{F}$ denote the space of bilinear forms on $\pxp$. By \cref{lem:correspndence_between_lines_and_quadrics}, there is an isomorphism between the set of quadrics passing through $p_1,p_2$, and lines $L\subset\textbf{F}$ that pass through $F_P$. The set of bilinear forms that are not of full rank is denoted by $\textbf{F}_2$. There are several different ways a line $l$ through $F_P$ can intersect $\textbf{F}_2$, namely:

When we have a finite number of intersections:
\begin{itemize}
\item Three distinct, real intersection points (one of which is $F_P$).
\item Three distinct intersection points, two of which are complex conjugates.
\item Two distinct intersection points of rank 2, $F_P$ has multiplicity 2.
\item Two distinct intersection points of rank 2, the one that is not $F_P$ has multiplicity 2.
\item Only one intersection point: $F_P$ with multiplicity 3.
\item Two distinct intersection points, $F_P$ and one which has rank 1. The latter will have multiplicity 2.
\end{itemize}
Then there are the ones where $L$ lies in $\textbf{F}_2$; here we look at the kernels, and at the number of rank 1 forms on $L$:
\begin{itemize}
\item Two distinct, real, rank 1 forms (implies all bilinear forms share the same left- and right kernel).
\item Two distinct, complex, rank 1 forms (implies all bilinear forms share the same left- and right kernel).
\item One form of rank 1, which has multiplicity 2 (implies all bilinear forms share the same left- and right kernel)
\item One form of rank 1, which has multiplicity 1 (implies all bilinear forms share the same left- OR right kernel, but never both)
\item No forms of rank 1, all bilinear forms share the same left- or right kernel
\item No forms of rank 1, all bilinear forms have distinct kernels.
\end{itemize}

This gives us a total of 12 different configurations to check. For each of these, we want to know what kind of quadric/camera configuration it corresponds to in $\p3$. 

Let $\PGL(3,\mathbb{R})$ be the projective general linear group of degree 3 (the group of real, invertible $3\times3$ matrices up to scale). This group acts on $\textbf{F}$ with an action that can be represented by matrix multiplication. Let $\PGL(3)_{F_P}$ be the subgroup of $\PGL(3)$ that fixes $F_P$. This gives us a group that acts on the $\p7$ of lines passing through $F_P$. This group will never take a line with one of the 12 configurations above to a line with a different configuration. For instance, a real invertible matrix can not take three distinct real points, to anything other than three distinct real points. In particular, this means that the 12 configurations above all lie in distinct orbits under this group action. 

Recall the isomorphism between the set of quadrics through $p_1,p_2$ and the set of lines through $F_P$. It gives us a similar group action on the set of quadrics passing through the camera centers $p_1,p_2$, namely the subgroup of $\PGL(4)$ that fixes the two camera centers. Here too, the camera/quadric configurations fall into 12 different orbits (not listed). Now we only need to check one representative from each orbit to find the exact correspondences.

By \cref{prop:only_camera_centers_matter}, the only property of the cameras that matters when considering critical configurations is their center. Furthermore, one pair of distinct points in $\p3$ is no different from any other pair. This means that when we make computations for the two view case, we are free to pick any two cameras $P_1,P_2$ with distinct centers. We use 
\begin{align*}
P_1=\begin{bmatrix}
1&0&0&0\\
0&1&0&0\\
0&0&1&0
\end{bmatrix},
&&P_2=\begin{bmatrix}
1&0&0&0\\
0&1&0&0\\
0&0&0&1
\end{bmatrix}.
\end{align*}
The fundamental matrix of these two cameras is 
\begin{align*}
F_P=\begin{bmatrix}
0&1&0\\
-1&0&0\\
0&0&0
\end{bmatrix}.
\end{align*}
Now we can go ahead and check the 12 possible configurations listed earlier. This will be done by picking a fundamental matrix $F_0$ such that the line $l$ spanned by $F_P$ and $F_0$ intersects $\textbf{F}_2$ the way we want and then checking what quadric it corresponds to. The full results are given in \cref{tab:appendix_table_large}.

\newpage
\begin{longtable}{
        @{}
        >{\raggedright\arraybackslash\begin{adjustbox}{valign=t}}p{0.13\textwidth}<{\end{adjustbox}}
        >{\raggedright\arraybackslash}p{0.45\textwidth}
        >{\raggedright\arraybackslash}p{0.36\textwidth}
        @{}
    }
    \caption{Possible intersections between $l$ and $\textbf{F}_2$, and their corresponding quadrics.}
    \label{tab:appendix_table_large}
    \\
    \toprule
    \bfseries\boldmath
    Matrix~$F_0$
    &
    \bfseries\boldmath
    Intersections between $l$ and $\mathbf{F}_2$
    &
    \bfseries\boldmath
    Configuration in $\p3$
    \\
    \midrule
    \endfirsthead
    \midrule
    \endhead
    \midrule
    \multicolumn{3}{@{}l}{Continues on the next page.}
    \endfoot
    \bottomrule
    \endlastfoot
    $\begin{bmatrix}
        0 & 0 & 0
        \\
        1 & 0 & 0
        \\
        0 & 0 & 1
    \end{bmatrix}$
    &
    Three distinct, real intersection points,
    (one of which is $F_P$)
    &
    Smooth ruled quadric, camera centers not on same line
    \\
    \addlinespace[.2cm]
    $\begin{bmatrix}
        1 & 0 & 0
        \\
        0 & 1 & 0
        \\
        0 & 0 & 1
    \end{bmatrix}$
    &
    Three distinct intersection points,
    two of which are complex conjugates.
    &
    Smooth, non-ruled quadric
    \\
    \addlinespace[.2cm]
    $\begin{bmatrix}
        1 & 0 & 0
        \\
        0 & 0 & 0
        \\
        0 & 0 & 1
    \end{bmatrix}$
    &
    Two distinct intersection points,
    the one that is not $F_P$ has multiplicity 2
    &
    A cone, cameras on different lines, no camera at vertex
    \\
    \addlinespace[.2cm]
    $\begin{bmatrix}
        0 & 0 & 0
        \\
        0 & 0 & 1
        \\
        1 & 0 & 0
    \end{bmatrix}$
    &
    Two distinct intersection points,
    $F_P$ has multiplicity 2
    &
    Smooth quadric, camera centers lie on the same line
    \\
    \addlinespace[.2cm]
    $\begin{bmatrix}
        0 & 0 &  1
        \\
        0 & 1 & 0
        \\
        1 & 0 & 0
    \end{bmatrix}$
    &
    Only one intersection point:
    $F_P$ with multiplicity 3
    &
    Cone, both cameras lie on the same line, neither lies at the vertex
    \\
    \addlinespace[.2cm]
    $\begin{bmatrix}
        0 & 0 & 0
        \\
        0 & 0 & 0
        \\
        0 & 0 & 1
    \end{bmatrix}$
    &
    Two distinct intersection points,
    $F_P$ and one which has rank 1.
    The latter has multiplicity 2
    &
    Union of two planes, camera centers in different planes
    \\
    \addlinespace[.2cm]
    $\begin{bmatrix}
        0 & 1 & 0
        \\
        0 & 0 & 0
        \\
        0 & 0 & 0
    \end{bmatrix}$
    &
    $l$ lies in $\mathbf{F}_2$,
    and contains two distinct, real, rank 1 forms\footnotemark
    &
    Union of two planes,
    camera centers lie on the intersection of the planes
    \\
    \addlinespace[.2cm]
    $\begin{bmatrix}
        1 & 0 & 0
        \\
        0 & 1 & 0
        \\
        0 & 0 & 0
    \end{bmatrix}$
    &
    $l$ lies in $\mathbf{F}_2$,
    and contains two distinct, complex, rank 1 forms\footnotemark[\value{footnote}]
    &
    Union of two complex conjugate planes,
    both cameras on intersection.
    \\
    \addlinespace[.2cm]
    $\begin{bmatrix}
        1 & 0 & 0
        \\
        0 & 0 & 0
        \\
        0 & 0 & 0
    \end{bmatrix}$
    &
    $l$ lies in $\mathbf{F}_2$,
    and contains one form of rank 1,
    which has multiplicity 2\footnotemark[\value{footnote}]
    &
    A single plane with multiplicity 2
    \\
    \addlinespace[.2cm]
    $\begin{bmatrix}
        0 & 0 & 1
        \\
        0 & 0 & 0
        \\
        0 & 0 & 0
    \end{bmatrix}$
    &
    $l$ lies in $\mathbf{F}_2$,
    and contains one form of rank 1,
    which has multiplicity 1\footnotemark
    &
    Union of two planes, one camera on the intersection
    \\
    \addlinespace[.2cm]
    $\begin{bmatrix}
        1 & 0 & 1
        \\
        1 & 1 & 0
        \\
        0 & 0 & 0
    \end{bmatrix}$
    &
    $l$ lies in $\mathbf{F}_2$,
    and contains no forms of rank 1,
    all bilinear forms share the same left or right kernel
    &
    Cone, one camera at the vertex
    \\
    \addlinespace[.2cm]
    $\begin{bmatrix}
        0 & 0 & 0
        \\
        0 & 0 & 1
        \\
        0 & 1 & 0
    \end{bmatrix}$
    &
    $l$ lies in $\mathbf{F}_2$,
    and contains no forms of rank 1,
    all bilinear forms have distinct kernels.
    &
    Union of two planes, cameras in same plane, neither at the intersection
\end{longtable}

\addtocounter{footnote}{-2}
 \stepcounter{footnote}\footnotetext{This implies all bilinear forms share the same left and right kernel}
 \stepcounter{footnote}\footnotetext{This implies all bilinear forms share the same left or right kernel}

\newpage
\addcontentsline{toc}{section}{References}
\bibliography{references.bib}

\end{document}